\documentclass[preprint]{elsarticle}

\usepackage{hyperref}
\usepackage{amsmath}
\usepackage{amsthm}
\usepackage{amsfonts}

\usepackage[T1]{fontenc}
\usepackage[utf8]{inputenc}
\usepackage{multirow}

\usepackage{listings}
\usepackage{color}

\definecolor{codegreen}{rgb}{0,0.4,0.1}
\definecolor{codegray}{rgb}{0.5,0.5,0.5}
\definecolor{codepurple}{rgb}{0.58,0,0.82}
\definecolor{codeblue}{rgb}{0,0,0.82}
\definecolor{backcolour}{rgb}{0.98,0.98,0.98}

\lstdefinestyle{mystyle}{
    backgroundcolor=\color{backcolour},
    commentstyle=\color{codegreen},
    keywordstyle=\color{blue},
    numberstyle=\tiny\color{black},
    stringstyle=\color{codepurple},
    basicstyle=\footnotesize,
    breakatwhitespace=false,
    breaklines=true,
    captionpos=b,
    keepspaces=true,
    numbers=left,
    numbersep=3pt,
    showspaces=false,
    showstringspaces=false,
    showtabs=false,
    tabsize=2
}

\newtheorem{thm}{Theorem}
\newtheorem{lem}{Lemma}
\newtheorem{prop}{Proposition}

\newdefinition{definition}{Definition}
\newdefinition{rem}{Remark}
\newdefinition{alg}{Algorithm}
\newtheorem{exmp}{Example}

\begin{document}

\begin{frontmatter}
\title{A unified explicit form for difference formulas for fractional and classical  derivatives}

\author[kdu]{W. A. ~Gunarathna}
\ead{anura@as.rjt.ac.lk}

\author[squ]{H. M. ~Nasir\corref{cor1}}
\ead{nasirh@squ.edu.om}
\ead{nasirhm11@yahoo.com}

\author[pdn]{W. B. ~Daundasekera}
\ead{wbd@pdn.ac.lk}

\address[kdu]{Department of Physical Sciences,
Rajarata University, Sri Lanka
}

\address[squ]{FracDiff Research Group, Department of Mathematics, Sultan Qaboos University,
Muscat, Sultanate of Oman}

\address[pdn]{Department of Mathematics,
University of Peradeniya,
Peradeniya, Sri Lanka}

\cortext[cor1]{Corresponding author}

\begin{abstract}
A unified explicit form for difference formulas to approximate the fractional and classical derivatives  is presented. The formula  gives finite difference approximations for any classical derivatives with a desired  order of accuracy at   nodal point in the computational domain. It also gives Gr\"unwald type approximations for fractional derivatives with arbitrary order of approximation at any point. Thus, this explicit  unifies approximations of both types of derivatives. Moreover, for classical derivatives, it  provides various finite difference formulas such as forward, backward, central, staggered, compact, non-compact etc. Efficient computations of the coefficients of the difference formulas are also presented that lead to automating the solution process of differential equations with a given higher order accuracy.  Some basic applications are presented to demonstrate the usefulness of this  unified formulation.
\end{abstract}


\end{frontmatter}

\renewcommand{\thefootnote}{\fnsymbol{footnote}}

\parindent 0pt
\parskip 5pt
\section{Introduction}

Fractional calculus has a history that goes back to L'Hospital,
Leibniz and Euler \cite{leibnitz1962letter, eulero1738progressionibus}.
A historical account of early works on fractional calculus can be found,
for eg., in \cite{ross1977development}.
Fractional integral and fractional derivative are extensions of the integer-order (or we call classical) integrals and derivatives to a real or complex order.
Various definitions of fractional derivatives have been proposed in the past, among which
the Riemann-Liouville, Gr\"unwald-Letnikov and Caputo derivatives are common and established.

Recently, fractional calculus found its way into the application domain in
science and engineering \cite{mainardi1996fractional,bagley1983theoretical,vinagre2000some,
metzler2004restaurant, hadian2020simulation}. Fractional derivative is also found to be more suitable to describe
anomalous transport in an external field
derived from the continuous time random walk \cite{barkai2000continuous},
resulting in a fractional diffusion equation (FDE).
The FDE involves fractional derivatives either in time, in space or in both  variables.

 Difference approximation in the form of an infinite series   is a widely used tool for numerically solving problems involving FDEs. The weights of the approximation formula  are obtained from the coefficients of a generator $ W_1(z) = (1-z)^\alpha $ expressed in a power series form \cite{podlubny1998fractional}. A shifted form of the difference approximation was proposed in \cite{meerschaert2004finite} to remedy some stability issues in solving space-fractional diffusion equations. Both difference approximations and its shifted form are known to be of the first-order accuracy (see Section \ref{Secpreliminary} for details).

 For higher order approximations, Lubich \cite{lubich1986discretized}  obtained some generators in the form of power of polynomials or rational polynomials. The coefficients of the expansion of the generators provide the coefficients for the  approximation with higher order accuracy without shift. Shifted forms with these coefficients give only the first order approximations regardless of their higher orders in non-shifted forms.

Recently, generators for higher order approximations with shifts have been obtained by Nasir and Nafa \cite{nasir2018new} and applied to  fractional diffusion equations with second-order accuracy, and a third-order approximation was obtained from the second order approximation and used for numerical approximations in \cite{NasirNafaANZIAM}.

The generators of these approximations are usually obtained manually by hand calculations, solving a system equations or by  symbolic computations and this processes are specific to the problem at hand.

However, most of the scientific researches  searching for mathematical patterns  are aimed at automating the relevant scientific processes so that the task may be performed for more general problems with minimal human intervention \cite{king2009automation}. In this regard, in numerical approximations and solution processes, especially in discretized numerical calculus, it is desirable to have general formulas for the approximation of derivatives that can be automated to be selected in  approximation processes.

The present authors \cite{gunarathna2019explicit} have obtained an explicit form for  generators that gives approximations for fractional derivatives with shifts retaining their higher orders. This form generalises the Lubich form with shift and hence the Lubich form  becomes a special case with no shift.

Interestingly,  this explicit form also gives coefficients for finite difference formulas (FDFs) for classical derivatives as demonstrated in \cite{gunarathna2019explicit}. However, it gives higher order compact FDFs for the first derivative only. Here, a compact FDF means the FDF which uses a minimum number of function values at the discrete grid points. For higher classical derivatives, the resulting FDFs are not compact although they are valid for FDFs.

Usually, the derivation of the weights for the FDFs for classical derivatives involves linear  combinations of the Taylor series expansions of a  function at various grid points about the point of derivative.  For higher order accuracy requirements, this leads to uncontrollable hand calculations, solving large system of linear equations or heavy symbolic computations \cite{khan2003taylor}. Besides, this technique is not suitable for fractional derivatives as the latter involves function values at infinite grid points, due to  its  non-local nature.

Explicit forms for finite difference formulas for classical derivatives have appeared in the past (see for eg. \cite{fornberg1988generation,fornberg1998classroom,hassan2012algorithm,
sadiq2014finite,zhang2018robust,lele1992compact,CameronTaylor,
bickley1941formulae,keller1978symbolic,rall1981automatic, aceto2002stable} ). However, all of them are focussed on the classical derivatives only.

In this paper, we extend the explicit form developed in \cite{gunarathna2019explicit} to a more  general unified explicit form that gives more new approximations for fractional derivatives and various finite difference formulas for any classical derivative.
The formulation also provides error coefficients of the approximation
for the first few terms.

We also present an algorithm to efficiently compute the  coefficients of the unified difference formulas. This enables one to compute the coefficients in exact  form where necessary and real coefficients with efficiency.

This paper is arranged as follows:
We start with preliminary preparations in Section  \ref{Secpreliminary}. The unified explicit  form is given in Section \ref{SecUnfiedExplictForm}.
Efficient computational strategies are described in Section \ref{Sec:Efficient}.
Various difference formulas for classical and fractional derivatives are given in Section \ref{SecFiniteDifferennceFormulas} demonstrating the unified nature of the explicit form. Applications of the difference formulas are presented in Section \ref{SecApplications}  and finally, Section \ref{SecConclusion} draws some conclusions.

\section{Preliminaries and terminologies}\label{Secpreliminary}
We list here relevant materials and define terminologies relation to the subject
of this paper.

Let $f(x)$ be a sufficiently smooth function defined on a real domain $\mathbb{R}$.
\begin{definition}
The left(-) and right(+) Riemann-Liouville (R-L) fractional derivatives of real
order $\alpha >0 $ are defined  as
\begin{equation}\label{Eq:LeftRL}
  \;^{RL} D_{x-}^\alpha f(x) = \frac{1}{\Gamma (n-\alpha) } \frac{d^n}{dx^n}
  \int_{-\infty}^{x} \frac{f(\eta)}{(x-\eta)^{\alpha + 1 - n}} d\eta ,
\end{equation}
and
\begin{equation}\label{Eq:RightRL}
  \;^{RL} D_{x+}^\alpha f(x) = \frac{(-1)^n }{\Gamma (n-\alpha) } \frac{d^n}{dx^n}
  \int_{x}^{\infty} \frac{f(\eta)}{(\eta-x)^{\alpha + 1 - n}} d\eta
\end{equation}
respectively, where $n = [ \alpha] + 1  $, an integer with
$n-1 < \alpha < n $ and $\Gamma (\cdot) $ denotes the gamma function.
\end{definition}
\begin{definition}
The left  Caputo fractional derivative of order $\alpha>0$ are defined analogously to (\ref{Eq:LeftRL}) by
\begin{equation}\label{LeftRL}
  \;^{C} D_{x-}^\alpha f(x) = \frac{1}{\Gamma (n-\alpha) }
  \int_{-\infty}^{x} \frac{f^{(n)} (\eta)}{(x-\eta)^{\alpha + 1 - n}} d\eta .
\end{equation}
The right Caputo derivative is defined analogously to (\ref{Eq:RightRL}).
\end{definition}
\begin{definition}
The left/right Gr\"unwald-Letnikov (G-L) definitions  are given respectively by
\begin{equation}\label{LeftGL}
  \;^{GL} D_{x\mp}^\alpha f(x) = \lim_{h \rightarrow 0} \frac{1}{h^\alpha } \sum_{k=0}^\infty (-1)^k \binom{\alpha}{k} f(x \mp k h),
\end{equation}
where the {\it Gr\"uwald weights},  $g_k^{(\alpha)}= (-1)^k \binom{\alpha}{k} = \frac{(-1)^k\Gamma(\alpha+1)}{k!\Gamma(\alpha-k+1)}$ are   the coefficients of the series expansion of the {\it Gr\"unwald generator}
$$W_{1}(z)=(1-z)^{\alpha}  = \sum_{k=0}^\infty g_k^{(\alpha)} z^k. $$
\end{definition}
These definitions of fractional derivatives are equivalent under certain smoothness conditions on $f(x)$ \cite{podlubny1998fractional}. Therefore, we denote them commonly as $D_{\mp x}^\alpha f(x)$ when the conditions are met.

\begin{definition}[Difference approximations]
For left and right fractional derivative of order $\alpha >0$, we define

\begin{enumerate}
\item
The {\it Gr\"unwald approximation}  (GA)
\begin{equation}\label{eq:No_shiftedGA}
\delta_{ \pm h}^\alpha f(x) = \frac{1}{h^\alpha } \sum_{k=0}^\infty g_k^{(\alpha)} f(x \mp k h),
\end{equation}
\item
The  {\it shifted Gr\"unwald approximation} (ShGA) \cite{meerschaert2004finite}  with shift $r$
\begin{equation}\label{eqshiftedGA}
\delta_{ \pm h,r}^\alpha f(x) = \frac{1}{h^\alpha } \sum_{k=0}^\infty g_k^{(\alpha)} f(x \mp (k-r)h),
\end{equation}
\end{enumerate}
where $h>0$ is the step size between discrete points in the domain referenced from $x$. Hence,  the difference approximations are defined on a {\it uniform grid } points  $x_k = x - kh, k\in \mathbb{Z} $.
\end{definition}

The approximation order of GA and ShGA is one for integer shifts \cite{meerschaert2004finite}.
\[
\delta_{ \pm,h,r}^\alpha f(x) = D_{x\mp }^\alpha f(x) + O(h).
\]
For a generalization of  the shifted  Gr\"unwald approximation, we define the following.
%
\begin{definition}
Let $w_k $ be a sequence of real numbers with it generating function
\[
 W(z) = \sum_{k=0}^\infty w_k z^k.
\]
Define a shifted difference formula
\begin{equation}\label{Eq:Difference_form}
\Delta_{\pm h,p,r}^{\alpha}f(x)=\frac{1}{h^{\alpha}}
\sum_{k=0}^{\infty}w_{k}^{(\alpha)}f(x\mp(k-r)h).
\end{equation}
We say that
\begin{enumerate}
\item
$W(z)$ approximates the  fractional derivatives
$D_{x\mp}^\alpha $  if
\begin{equation}\label{Eq:Approx}
 \lim_{h\rightarrow 0} \Delta_{\pm h,p,r}^{\alpha}f(x) =  D_{x\mp}^\alpha f(x),
\end{equation}
\item
 $W(z)$ approximates the   fractional derivatives
$D_{x\mp}^\alpha $ with order $p$ if
\begin{equation}\label{Eq:Approx_order_p}
\lim_{h\rightarrow 0} \Delta_{\pm h,p,r}^{\alpha}f(x) = D_{x\mp}^\alpha f(x) + O(h^p).
\end{equation}
\end{enumerate}
We call these approximations (\ref{Eq:Approx}) and (\ref{Eq:Approx_order_p})   the {\it Gr\"unwald type approximations} (GTAs), $W(z)$ the {\it Gr\"unwald type generator} (GTG) and $w_k$ the {\it Gr\"unwald type weights} (GTWs).
\end{definition}

 A set of GTGs for higher order GTAs was established in Lubich \cite{lubich1986discretized} by utilizing the characteristic polynomials of linear multistep methods (LMMs) for classical initial value problems. Particularly, a set of   generators in the form of
$W_p(z) = (P_p(z))^\alpha $ was obtained from the backward difference LMM for orders $ 1\le p \le 6$ and is listed in  Table \ref{Tab:Lubich}. For $p = 1$, it gives the Gr\"unwald generator $ W_1(z)   $.

%
%
\begin{table}[h] \centering
\begin{tabular}{ll|ll}
\hline
\hline
$p$  &  $ P_p(z) $ & $p$  &  $ P_p(z) $\\
\hline \hline
&&&\\
1 &  $ 1 - z  $   &
4  &  $  \frac{25}{12} - 4z + 3z^2 -\frac{4}{3}z^3 + \frac{1}{4} z^4 $ \\
&&&\\
2 &  $   \frac{3}{2} -2z +\frac{3}{2}z^2  $ &
5 &  $    \frac{137}{60} - 5z + 5z^2 -\frac{10}{3}z^3 + \frac{5}{4} z^4 - \frac{1}{5} z^5  $\\&&&\\
3 &  $  \frac{11}{6} - 3z + \frac{3}{2}z^2 -\frac{1}{3}z^3   $ &
6  &  $  \frac{147}{60} - 6z + \frac{15}{2}z^2 -\frac{20}{3}z^3 + \frac{15}{4} z^4 - \frac{6}{5} z^5 + \frac{1}{6}z^6   $\\&&&\\
\hline
\end{tabular}
\caption{Polynomials $ P_p(z) $ for Lubich generators $W_p(z) = (P_p(z))^\alpha $} \label{Tab:Lubich}
\end{table}

\begin{definition}
We call the generators given in Table \ref{Tab:Lubich} for orders $ 2 \le p \le 6$
as the {\it Lubich generators} and their weights the {\it Lubich weights}.
\end{definition}

The Lubich generators are to be used without shift to obtain their respective
higher approximation orders. Applying the Lubich weights with non-zero shifts reduces the approximation orders to one \cite{nasir2018new}, \cite{gunarathna2019explicit},\cite{NasirNafa}.

Nasir and Nafa \cite{NasirNafa}   obtained a generator $W_{2,r} (z) $ for an order 2 approximation whose weights can be used with shift without reducing the order:
\begin{equation}\label{Eq:W2_r}
  W_{2,r}(z) = \left(  \left(\frac{3}{2} - \frac{r}{\alpha} \right)
  + \left(-2 + \frac{2r}{\alpha} \right)  z + \left(\frac{1}{2} - \frac{r}{\alpha} \right) \right)^\alpha.
\end{equation}

An equivalent characterization for a GTA of order $p$ with shift  $r$   was  established by the  authors in \cite{NasirNafa} and is given in Proposition \ref{propG(z)Approxi}.

\begin{prop}\label{propG(z)Approxi} (Theorem 1 \cite{NasirNafaANZIAM, NasirNafa}): Let $\alpha>0, n=[\alpha]+1,  $  a non-negative integer $m$ be given. Let a function $f(x)\in C^{m+n+1}(\mathbb{R})$ and $D^{k}f(x)=\frac{d^{k}}{dx^{k}}f(x)\in L_{1}(\mathbb{R})$ for $0 \leq k \leq m+n+1$. Then, a generator $W(z)$ approximates the   fractional derivatives $D_{x\pm}^{\alpha}f(x)$ with order $p$ and shift $r$, $1\leq p \leq m$, if and only if
	\begin{equation}\label{eqG(z)}
	G(z)=\frac{1}{z^{\alpha}}W(e^{-z})e^{rz}=1+O(z^{p}).
	\end{equation}
	
Moreover, if $G(z)=1+\sum_{l=p}^{\infty}a_{l}z^{l}$, where $a_{l} \equiv a_{l}(\alpha, r)$, then we have
\begin{align}\label{eqTaylorfractional}
\Delta_{\pm h,p,r}^{\alpha}f(x) &= D_{x\pm}^{\alpha}f(x)+h^{p}a_{p}D_{x\pm}^{\alpha+p}f(x)
   + h^{p+1}a_{p+1}D_{x\pm}^{\alpha+p+1}f(x)+\cdots   \nonumber \\
   &+ h^{m}a_{m}D_{x\pm}^{\alpha+m}f(x) +  O(h^{m+1}).
\end{align}
\end{prop}

This characterization theorem was used in \cite{gunarathna2019explicit} to obtain an explicit form for generators of the form $W(z)=\left(\beta_{0}+\beta_{1}z+\cdots+\beta_{p}z^{p}\right)^{\alpha}$ for an approximation of order $p$  with shift $r$.

\begin{prop}\label{explictform} (Theorem 3 \cite{gunarathna2019explicit})
The generator of the form
\begin{equation}\label{eq:LubichtypeForm}
W_{p,r}(z) = (\beta_0 + \beta_1 z +\cdots + \beta_p z^p)^\alpha
\end{equation}
approximates the   fractional derivatives $ D_{x\mp}^\alpha f(x)$ with order $p$ and shift $r$   if and only if the coefficients $\beta_j$  are given by
\begin{align}\label{eq:ExplictBetaj}
\beta_{j}=-\left(\sum_{\begin{smallmatrix}
	m=0\\m\ne j
	\end{smallmatrix}}^{p} \prod_{\begin{smallmatrix}
	l=0\\l\ne m,j
	\end{smallmatrix}}^{p}(\lambda-j)\right)\left(\prod_{\begin{smallmatrix}
	m=0 \\ m\ne j
	\end{smallmatrix}}^{p}\frac{1}{j-m}\right), j=0, 1, 2, \cdots, p,
\end{align}
where $\lambda=r/\alpha$, and the leading error coefficient  is given by

\begin{equation}\label{eqErrorCont}
R=\frac{1}{(p+1)}\sum_{j=0}^{p}(\lambda-j)^{p+1}\beta_{j}.
\end{equation}
\end{prop}

Using this explicit form, the authors in \cite{gunarathna2019explicit} have constructed generators of the form (\ref{eq:LubichtypeForm}), for orders $ 1\le p \le 6 $ with shifts and  the first five of the generators are listed in Table \ref{Tab:Lubich_Type}.

\begin{table}[h] \centering
\begin{tabular}{lll}
\hline
\hline
$p$ &   Coefficients $ \beta_j , \quad 0 \le j \le p $  & with $ \lambda = r/\alpha $\\
\hline\hline
1 & $\beta_0= 1 $ &
 $\beta_1= -1 $\\
\hline
2 & $\beta_0= - \lambda + \frac{3}{2} $ &
 $\beta_1= 2 \lambda - 2 $\\
 & $\beta_2= - \lambda + \frac{1}{2} $\\
\hline
3 & $\beta_0= \frac{\lambda^{2}}{2} - 2 \lambda + \frac{11}{6} $ &
 $\beta_1= - \frac{3 \lambda^{2}}{2} + 5 \lambda - 3 $\\
& $\beta_2= \frac{3 \lambda^{2}}{2} - 4 \lambda + \frac{3}{2} $
& $\beta_3= - \frac{\lambda^{2}}{2} + \lambda - \frac{1}{3} $\\
\hline
4 & $\beta_0= - \frac{\lambda^{3}}{6} + \frac{5 \lambda^{2}}{4} - \frac{35 \lambda}{12} + \frac{25}{12} $
 & $\beta_1= \frac{2 \lambda^{3}}{3} - \frac{9 \lambda^{2}}{2} + \frac{26 \lambda}{3} - 4 $\\
 & $\beta_2= - \lambda^{3} + 6 \lambda^{2} - \frac{19 \lambda}{2} + 3 $
 & $\beta_3= \frac{2 \lambda^{3}}{3} - \frac{7 \lambda^{2}}{2} + \frac{14 \lambda}{3} - \frac{4}{3} $\\
 & $\beta_4= - \frac{\lambda^{3}}{6} + \frac{3 \lambda^{2}}{4} - \frac{11 \lambda}{12} + \frac{1}{4} $\\
\hline
5 & $\beta_0= \frac{\lambda^{4}}{24} - \frac{\lambda^{3}}{2} + \frac{17 \lambda^{2}}{8} - \frac{15 \lambda}{4} + \frac{137}{60} $
 & $\beta_1= - \frac{5 \lambda^{4}}{24} + \frac{7 \lambda^{3}}{3} - \frac{71 \lambda^{2}}{8} + \frac{77 \lambda}{6} - 5 $\\
 & $\beta_2= \frac{5 \lambda^{4}}{12} - \frac{13 \lambda^{3}}{3} + \frac{59 \lambda^{2}}{4} - \frac{107 \lambda}{6} + 5 $
 & $\beta_3= - \frac{5 \lambda^{4}}{12} + 4 \lambda^{3} - \frac{49 \lambda^{2}}{4} + 13 \lambda - \frac{10}{3} $\\
 & $\beta_4= \frac{5 \lambda^{4}}{24} - \frac{11 \lambda^{3}}{6} + \frac{41 \lambda^{2}}{8} - \frac{61 \lambda}{12} + \frac{5}{4} $
 & $\beta_5= - \frac{\lambda^{4}}{24} + \frac{\lambda^{3}}{3} - \frac{7 \lambda^{2}}{8} + \frac{5 \lambda}{6} - \frac{1}{5} $\\
 %
 %
%
%
\hline
\hline
\end{tabular}
\caption{Coefficients for   $ p(z) = \beta_0 + \beta_1 z+\cdots +\beta_{p} z^{p}   $ for generator $ W_{p,r}(z) = (p(z))^{\alpha } $  } \label{Tab:Lubich_Type}
\end{table}

When there is no shift ($r = 0$), these   generators $ W_{p,0}(z) $ reduce  to
Lubich generators in Table \ref{Tab:Lubich}.
Moreover, for order $p =1$ it recides to the Gr\"unwald generator $W_1(z)$.
\begin{definition}
We call the generators given in Table \ref{Tab:Lubich_Type} for orders $ 2 \le p \le 6 $
 the {\it Lubich type generators} and their weights the {\it Lubich type weights}.
\end{definition}

We mention that the Lubich type generators also give finite difference formulas for classical derivatives with integer derivative order $\alpha$. However, we point out that the Lubich type operators give compact finite difference formula for the first order derivative only. For other higher derivatives, we get valid finite difference formulas with non-compact form.

\section{A unified explicit form}\label{SecUnfiedExplictForm}

In this section, we extend  the explicit form (\ref{eq:ExplictBetaj}) appeared in \cite{gunarathna2019explicit}  to a more general form that covers compact finite difference formulas for higher order classical derivatives as well and some new Lubich type generators  for fractional derivatives.

For this, we introduce a base differential order $d$, a positive integer, to express the fractional derivative as $$ D_{x\pm}^{\alpha}=\left(D^{d}\right)_{x\pm}^{\frac{\alpha}{d}}, $$  and consider approximating the fractional derivative by a Lubich type generator of the form
\begin{equation}\label{eqW(z)d}
W(z)=\left(\beta_{0}+\beta_{1}z+\cdots+\beta_{N}z^{N}\right)^{\frac{\alpha}{d}}=\left(P(z)\right)^\frac{\alpha}{d},
\end{equation}
where $P(z)$ corresponds to the classical derivative operator $D^{d}$.

The coefficients $\beta_{j}$  in (\ref{eqW(z)d}) are to be determined based on the fractional order $\alpha$, the required approximation order $p$, and shift $r$ for the GTA. The degree $N$ of $P(z)$ will also be determined according to the choice of $p$ and $d$. Then, we have the following theorem leading to our main result.
\begin{thm}\label{ThmLinearSystem}
With assumptions of Proposition \ref{propG(z)Approxi},
	the generator of the form $W(z)=\left(\beta_{0}+\beta_{1}z+\cdots+\beta_{N-1}z^{N-1}\right)^{\frac{\alpha}{d}} $, where $d$ is a positive integer,  approximates the  fractional derivatives
	$D_{x\mp}f(x)$ at $x$ with order $p$  and shift $r$  if and only if
	the coefficients   $ \beta_j $ satisfy the linear system
	\begin{equation}\label{eqVanderMonde}
	\sum_{j=0}^{N-1} (\lambda - j)^k \beta_j = d!\delta_{d,k}, \qquad k = 0,1,\cdots, N-1,
	\end{equation}
	where $\lambda = rd/\alpha, \quad N=p+d$ and $\delta_{d,k}$ is the Kronecker delta having value of one for  $k=d$ and zero otherwise.
\end{thm}
\begin{proof}
	In view of Proposition \ref{propG(z)Approxi}, we have $G(z)=\frac{1}{z^{\alpha}}W\left(e^{-z}\right)e^{rz}=1+O(z^{p})$.
	This gives
	\begin{align*}
	G(z) &  =
	\frac{1}{z^\alpha} \left(\sum_{j=0}^{N} \beta_j e^{-jz}\right)^{\frac{\alpha}{d}} e^{rz} =
	\frac{1}{z^\alpha} \left(\sum_{j=0}^{N} \beta_j e^{(rd/\alpha -j)z}\right)^{\frac{\alpha}{d}} =\left(\frac{1}{z^{d}}\sum_{j=0}^{N}\beta_{j}e^{\lambda_{j}z}\right)^{\frac{\alpha}{d}}\\
	&=\left(\frac{1}{z^{d}}\sum_{j=0}^{N}\beta_{j}\sum_{k=0}^{\infty}\frac{1}{k!}\lambda_{j}^{k}z^{k}\right)^{\frac{\alpha}{d}}=\left(\frac{1}{z^{d}}\sum_{k=0}^{\infty}b_{k}z^{k}\right)^{\frac{\alpha}{d}}\\
	&=\left(\frac{b_{0}}{z^{d}}+\frac{b_{1}}{z^{d-1}}+\cdots+\frac{b_{d-1}}{z}+b_{d}+\sum_{k=d+1}^{\infty}b_{k}z^{k-d}\right)^{\frac{\alpha}{d}}=1+O(z^{p}),
	\end{align*}
	where $\lambda_{j}=\lambda-j,\quad \lambda=\frac{rd}{\alpha}$ and
	\begin{equation}\label{Eq_Coefficient}
	b_k = \frac{1}{k!} \sum_{j=0}^{N} \lambda_j^{k} \beta_j ,
	\qquad n = 0,1,2,\cdots.
	\end{equation}
	
Since $G(z)$   does not have any pole singularities by virtue of (\ref{eqG(z)}), we have $b_k = 0$ for $k=0, 1, \cdots, d-1$.
 Moreover, since $G(0)=1$, we have $b_{d}=1$. These are the consistency condition for the GTA with generator $W(z)$.
Now, for order $p=1$, these  conditions give the system
	(\ref{eqVanderMonde}) with $N = 1+d$ and the proof ends.

For $p>1$, $G(z)$ reduces to

\begin{equation*}\label{eqConditionOrderp}
G(z) = \left(
1+ \sum_{k=d+1}^\infty b_k z^{k-d} \right)^{\frac{\alpha}{d}}=: (1+X)^{\gamma}=1+O(z^p),
\end{equation*}
where $\gamma=\frac{\alpha}{d}$ and
\begin{equation}\label{Eq_X}
X =  \sum_{k=d+1}^\infty b_k z^{k-d}.
\end{equation}
Expansion of $(1+X)^{\gamma}$  gives
\begin{equation}\label{eqTaylorExpansion1+x}
1+\gamma X+\frac{\gamma(\gamma-1)}{2!}X^{2}+\cdots=1+O(z^{p}).
\end{equation}
The term with $z$ appears in the term $\gamma X$ only on the left hand side of (\ref{eqTaylorExpansion1+x}). This gives $b_{d+1}=0$. The same is true for   $b_k$ by successively comparing the coefficients of $z^{k-d}, d<k< p+d$, to gain $O(z^{p})$ in (\ref{eqTaylorExpansion1+x}).

Altogether, we have $b_k=\delta_{d,k}$, $k=0, 1, 2, \cdots, p+d-1$ which yield the linear system(\ref{eqVanderMonde}) with (\ref{Eq_Coefficient}) and $ N = p+d$.
\end{proof}
The linear system (\ref{eqVanderMonde}) can be expressed in matrix form
\begin{equation}\label{eqVandermondMatrixform}
V(\lambda_0, \lambda_1, \cdots, \lambda_{N-1})\mathbf{b}=\mathbf{d},
\end{equation}
with $ \mathbf{b}=\left[\beta_0, \beta_1, \cdots, \beta_{N-1}\right]^{T}$, $\mathbf{d} =[0, 0, \cdots, d!, 0, \cdots,0 ]^{T}$, where $d!$ is at the $d^{th}$  position, and the Vandermonde matrix $V \equiv V (\lambda_0, \lambda_1, \cdots, \lambda_{N-1})$ whose   columns are  $[1, \lambda_k, \lambda_k^2, \cdots, \lambda_{k}^{N-1}]^{T},  k = 0,1,\cdots,N-1,$ with its determinant
\begin{equation}\label{Eq:Vande_Determinat}
|V|=\prod_{0\le n<m\le N-1}(\lambda_m-\lambda_n).
\end{equation}

Our main result is that the coefficients $\beta_j , 0 \le j <N $ of ${\bf b} $ in (\ref{eqVandermondMatrixform}) can be explicitly expressed as  given in the following theorem.
\begin{thm}\label{ThmBetaSolution}
Let $\alpha >0$, a positive integer $d\ge 1$ and $f(x)$ be a sufficiently smooth function such that   $D_{x\pm}^{\alpha}f(x)$ is defined. For a Gr\"unwald type approximation (\ref{Eq:Difference_form}) for $D_{x \pm}^{\alpha}f(x)$ of order $p$ and shift $r$ with the generator in the form (\ref{eqW(z)d}), the coefficients $\beta_{j}$ are given by
	
	\begin{equation}\label{eqnBetaNumerator-Denominator}
	\beta_{j}=\frac{N_{j}}{D_{j}}, \qquad j=0, 1,\cdots, N-1,
	\end{equation}
	
where $N=p+d$ and
\begin{equation}\label{Eq_NjDj}	
N_{j}=\sum_{\begin{smallmatrix}
		0\leq m_1<m_2<\cdots<m_{p-1}\leq N-1\\m_i \ne j, 1\le i \le p-1
		\end{smallmatrix}} \prod_{k=0}^{p-1}(\lambda-m_k), \quad D_{j} =\frac{(-1)^d}{d!} \prod_{\begin{smallmatrix}
		m=0\\m \ne j
		\end{smallmatrix}}^{N-1}(j-m).
\end{equation}
Moreover, the leading and some successive error coefficients $ R_m $ of the approximation are given by
\begin{equation}\label{eq_ErrorR(d)}
R_m=\frac{\alpha}{m! d}\sum_{j=0}^{N-1}(\lambda-j)^{m} \beta_j,\quad
m = N, N+1, \cdots, N + p.
\end{equation}
\end{thm}

\begin{proof}
Solving the Vandermonde system (\ref{eqVandermondMatrixform}) by Cramer's rule, we get
	
	\begin{equation}\label{eqBetabyCramerRule}
	\beta_j=\frac{|V_{j}(\mathbf{d})|}{|V|}, \qquad j=0, 1, \cdots, N-1,
	\end{equation}
where $|V| = |V(\lambda_0, \lambda_1, \cdots, \lambda_{N-1})| $ and $  V_{j}(\mathbf{d})$ is the matrix obtained from $V$ by replacing its $j^{\text{th}}$ column with $\mathbf{d}$.

For evaluating the determinant $|V_j(\mathbf{d})|$, observe that the $j^{\text{th}}$ column has only one non-zero entry $ d! $ at the $d^{\text{th}}$ place. Hence, pivoting with the $j^{\text{th}}$ column, we have $|V_{j}(\mathbf{d})|=(-1)^{j+d}|V_{j,d}|$, where $|V_{j,d}|$ is the determinant of the matrix obtained from removing the $j^{\text{th}}$ column and $d^{\text{th}}$ row from   $V$, and is given by (see \ref{App:Vander})
\[
|V_{j}(\mathbf{d})|=(-1)^{j+d}d!\prod_{\begin{smallmatrix}
	m>n\\ m,n \ne j
	\end{smallmatrix}}(\lambda_{m}-\lambda_{n})N_j.
\]
Rearranging the Vandermonde determinant (\ref{Eq:Vande_Determinat}) as
\[ |V|=\prod_{\begin{smallmatrix}
	0\le n<m\le N-1 \\ m, n \ne j
	\end{smallmatrix}}(\lambda_{m}-\lambda_{n})(-1)^j \prod_{\begin{smallmatrix}
	m=0\\m\ne j
	\end{smallmatrix}}^{N-1}(\lambda_m-\lambda_j)
\]
and substituting in (\ref{eqBetabyCramerRule}), we have (\ref{eqnBetaNumerator-Denominator})  with $\lambda_m-\lambda_j=j-m$.

For the leading and subsequent error terms,  note that (\ref{Eq_X}) is reduced to
\begin{equation*}
X = b_{N}z^{p} +b_{N+1}z^{p+1} + \cdots + b_{N+p-1}z^{2p-1} +\cdots.
\end{equation*}
In (\ref{eqTaylorExpansion1+x}), the   terms with $z^{k} , k = p, p+1, \cdots, 2p-1$ appear only in $X$.
Thus,  the error coefficients are given by $$ R_m =  \gamma b_m=\frac{\alpha}{m!d}\sum_{j=0}^{N-1}(\lambda-j)^{m}\beta_j,\qquad
 N \le m < N+p. $$
\end{proof}

The numerator and denominator terms of the coefficients $\beta_j$ can be easily and efficiently evaluated despite their complicated expressions. Moreover, for rational  parameters of $\alpha,  p $  and   $r$, these terms will also be rational numbers. Therefore, finite difference coefficients for classical derivatives
can be evaluated to their exact rational values.

\section{Efficient evaluations}\label{Sec:Efficient}

We describe the details of evaluating the coefficients of the unified formulation.
The denominators and numerators terms of the coefficients  $\beta_j$ are evaluated separately.

Separation of the evaluation of numerators and denominators allows one to implement the
explicit expression efficiently and to express the coefficients in
exact fractional form where possible, that is, when $\lambda$ is an integer or a rational number.

\subsection{Evaluating the denominator }

The denominator  is evaluated in recursive form
for efficient computation as follows:
The denominator of the coefficient $\beta_j$   given in (\ref{Eq_NjDj}) can be expressed as
\[
D_j = \frac{(-1)^{d}} {d!}\prod_{m=0}^{N-1} (j-m) = \frac{(-1)^{p-1-j}} {d!} j!  (N-1-j)!.
\]

This can be described by the recursive form
\[
D_0 = \prod_{m=d+1}^{N-1} (-m)
,\qquad  D_j = \frac{-j}{N-j}D_{j-1}, j = 1,2,\cdots, N-1.
\]

Hence, the algorithm for the denominator is as follows:
\begin{alg} Compute the denominators
\begin{enumerate}
\item Compute $D_0$:\\
 $ D_0 \leftarrow -m D_0, \quad \text{ for } m = d+1, d+2, \cdots, N-1. $
\item Compute $D_j$ :\\

$ D_j \leftarrow \frac{-j}{N-j}D_{j-1},\qquad $ for $ j = 1, 2, \cdots, N - 1$.
\end{enumerate}
\end{alg}

Since the denominator is independent of any shift $r$ and derivative order $\alpha$, it can be evaluated once and used for various shifts and order $\alpha$.

\subsection{Evaluating the numerator }

An efficient algorithm for evaluating the numerators can be devised  once we identify the sum-product term in (\ref{Eq_NjDj}) as elementary symmetric polynomials (ESPs) $S(X_j, p-1)$ in the set
\[
X_j = \{ \lambda_k = \lambda - k :  \quad  0\le k \le N-1, k\ne j \}.
\]

 The ESPs can be
efficiently computed by the recursive relations described in  \ref{App:ESP}.
We use the  operations (\ref{addfactor}) and (\ref{removefactor}) to obtain the required coefficients
$ S(X_j,p-1), j = 0,1 ,\cdots, N-1$.

Note that each numerator in $\beta_j$ has the ESP on distinct sets $A_j$.
In fact, $S(X_j, p-1) $ is the coefficient of $x^{N-1-(p-1)} = x^{d} $ in the monomial
\[
M_j(x) = \prod_{m = 0, m \ne j}^{N-1} (x + x_m)
\]
of degree $N-1$ that can be expressed in the recursive form
\[
M_0(x) = \prod_{m = 1}^{N-1} (x + x_m), \qquad
M_j(x) = M_{j-1}(x) \frac{x + x_{j-1}}{x + x_j}.
\]

The polynomials $M_j(x)$  can be computed by combining the two recurrence relations in (\ref{addfactor}) and (\ref{removefactor}) as follows:

Let
\[
M_{j-1} (x) = \sum_{k=0}^{N-1} p_k x^k, \qquad
M_j(x) = \sum_{k=0}^{N-1} q_k x^k.
\]
and denote their coefficient vectors as $ {\bf p} $ and  $ {\bf q} $ respectively.
Then, the algorithm to compute the numerators is devised as follows:
\begin{alg}
Computing the numerator terms $N_j$ .
\begin{enumerate}
\item Compute $ M_0(x)  $:

$ p_0 = 1, $\\
For $ m = 1 , 2, \cdots, N-1 $, with $ p_{-1} = 0 $,
\[
p_k \leftarrow  p_{k-1} + x_m p_k , \qquad k = m,m-1, \cdots, 0.
\]
$ N_0 =  S(X_0, p-1) = p_d.$

\item Compute $ M_j(x) $ :

For $j = 1, 2, \cdots , N-1$, with $p_N = q_N = 0$,
\begin{align*}
q_k &\leftarrow  p_k + x_{j-1}p_{k+1} -x_j q_{k+1}, \text{ for }  k =  N-1, N-2, \cdots, 0. \\
{\bf p} &\leftarrow  {\bf q},
\end{align*}
$ N_j = S(X_j , p-1)   = q_{d}. $
\end{enumerate}
\end{alg}

\subsection{Evaluation of Error}
The leading and subsequent error coefficients are evaluated
directly as the expression involves on a sum. From (\ref{eq_ErrorR(d)}), we have the following algorithm:

\begin{alg} Computing the error coefficient of order $m, p \le m \le 2p $.

Given $\beta_j = N_j/D_j $ for $ j = 0,1,\cdots,N-1 $,
\begin{enumerate}
\item $ E \leftarrow 0$, \\
For $ j = 0,1,\cdots, N-1, $
$
    E \leftarrow  E + \lambda_j^{m+d}\beta_j,
$
\item $ R_m \leftarrow  \frac{\alpha}{(m+d)! d} E .$
\end{enumerate}
\end{alg}

\subsection{Operation count}

The direct evaluation of the numerator part in (\ref{Eq_NjDj})
computes the sum-product form with all possible combinations. The operation counts for computing the numerator part with base differential order $d$ and approximation order $p$ are
$MN$ additions and $MN(p-1)$ multiplications, where $ M = \binom{N-1}{p-1} $ and $N=p+d$.
For small values of $N$, this count is reasonably small. However,
for large $N$, the operation count uncontrollably increases. For example, the number of additions and multiplications for computing the numerator
for a 10-th order derivative with approximation order 10 by FDNum are  1847560   and 16628040 respectively.

However, the efficient computation described above
computes the same numerator in only 968 additions and the same amount of multiplications.

\section{Difference formulas}\label{SecFiniteDifferennceFormulas}

We use the unified explicit form in this paper to  construct
approximating Lubich type generators with a given base differential order for the left and right fractional derivatives along with their leading
error coefficients. We also list some finite difference
formulas of various types that can be constructed from the unified formulation.

\subsection{Difference approximations for fractional derivatives}\label{SecFracGens}

When the base differential operator  is the first derivative operator  $D = d/dx $ with base differential order $ d = 1$, we obtain the Lubich type generators obtained in \cite{NasirNafa} with shift $r$ and are given in Table \ref{Tab:Lubich_Type}. When there is no shift ($r=0$), this gives the Lubich generators in Table \ref{Tab:Lubich}. Hence, our unified formulation is
a further generalization of the  Lubich type generators and the Lubich generators for the approximation of fractional derivatives with and without shifts respectively.

When the base differential operator is of order $d>1$, we get new approximating generators.
Tables \ref{TableGens2} and  \ref{TableGens3} list some  approximating
generators with base differential orders $d = 2$ and 3 respectively and approximation orders $p = 1,2,3,4,5$. Note that the first order approximations in all cases
turn out to be the same Gr\"unwald generator $W_1(z) = (1-z)^\alpha $.

The coefficients $w_k  $ of the generators can be computed from the J.C.P. Miller formula given as follows (see \cite{weilbeer2006efficient} for a proof):
\[
W(z) = \left( \sum_{j=0}^{N-1} \beta_j z^j \right)^\gamma
= \sum_{k=0}^{\infty} w_k z^k
\]
\begin{equation}\label{Eq:Miller}
w_0  = \beta_0^\gamma ,\quad
w_m   = \frac{1}{m w_0 }
    \sum_{k=1}^{M}( k (\gamma+1)-m) \beta_k w_{m-k} ,\quad
    m = 1, 2, \cdots,
\end{equation}
where $ M = \min{(m, N)}.$

The Gr\"unwald type approximation for the left or right fractional derivative is then given by (\ref{Eq:Difference_form}).

\begin{table}\centering
\begin{tabular}{lll}
\hline
\hline
$p$ & Coefficients $ \beta_j , 0\le j \le p+d-1 $ & with  $ d = 2 $ and $ \lambda = rd/\alpha $ \\
\hline
\hline
1 & $\beta_0= 1 $
 & $\beta_1= -2 $\\
 & $\beta_2= 1 $\\
\hline
2 & $\beta_0= - \lambda + 2 $
 & $\beta_1= 3 \lambda - 5 $\\
 & $\beta_2= - 3 \lambda + 4 $
 & $\beta_3= \lambda - 1 $\\
\hline
3 & $\beta_0= \frac{\lambda^{2}}{2} - \frac{5 \lambda}{2} + \frac{35}{12} $
 & $\beta_1= - 2 \lambda^{2} + 9 \lambda - \frac{26}{3} $\\
 & $\beta_2= 3 \lambda^{2} - 12 \lambda + \frac{19}{2} $
 & $\beta_3= - 2 \lambda^{2} + 7 \lambda - \frac{14}{3} $\\
 & $\beta_4= \frac{\lambda^{2}}{2} - \frac{3 \lambda}{2} + \frac{11}{12} $\\
\hline
4 & $\beta_0= - \frac{\lambda^{3}}{6} + \frac{3 \lambda^{2}}{2} - \frac{17 \lambda}{4} + \frac{15}{4} $
 & $\beta_1= \frac{5 \lambda^{3}}{6} - 7 \lambda^{2} + \frac{71 \lambda}{4} - \frac{77}{6} $\\
 & $\beta_2= - \frac{5 \lambda^{3}}{3} + 13 \lambda^{2} - \frac{59 \lambda}{2} + \frac{107}{6} $
 & $\beta_3= \frac{5 \lambda^{3}}{3} - 12 \lambda^{2} + \frac{49 \lambda}{2} - 13 $\\
 & $\beta_4= - \frac{5 \lambda^{3}}{6} + \frac{11 \lambda^{2}}{2} - \frac{41 \lambda}{4} + \frac{61}{12} $
& $\beta_5= \frac{\lambda^{3}}{6} - \lambda^{2} + \frac{7 \lambda}{4} - \frac{5}{6} $\\
\hline
5 & $\beta_0= \frac{\lambda^{4}}{24} - \frac{7 \lambda^{3}}{12} + \frac{35 \lambda^{2}}{12} - \frac{49 \lambda}{8} + \frac{203}{45} $
 & $\beta_1= - \frac{\lambda^{4}}{4} + \frac{10 \lambda^{3}}{3} - \frac{31 \lambda^{2}}{2} + 29 \lambda - \frac{87}{5} $\\
 & $\beta_2= \frac{5 \lambda^{4}}{8} - \frac{95 \lambda^{3}}{12} + \frac{137 \lambda^{2}}{4} - \frac{461 \lambda}{8} + \frac{117}{4} $
 & $\beta_3= - \frac{5 \lambda^{4}}{6} + 10 \lambda^{3} - \frac{121 \lambda^{2}}{3} + 62 \lambda - \frac{254}{9} $\\
 & $\beta_4= \frac{5 \lambda^{4}}{8} - \frac{85 \lambda^{3}}{12} + \frac{107 \lambda^{2}}{4} - \frac{307 \lambda}{8} + \frac{33}{2} $
 & $\beta_5= - \frac{\lambda^{4}}{4} + \frac{8 \lambda^{3}}{3} - \frac{19 \lambda^{2}}{2} + 13 \lambda - \frac{27}{5} $\\
 & $\beta_6= \frac{\lambda^{4}}{24} - \frac{5 \lambda^{3}}{12} + \frac{17 \lambda^{2}}{12} - \frac{15 \lambda}{8} + \frac{137}{180} $\\
\hline
\hline
\end{tabular}
\caption{Coefficients for polynomial $ p(z) = (\beta_0 + \beta_1 z+\cdot +\beta_{p+d-1} z^{p+d-1}  ) $ for generator $ W_{p,d}(z) = (p(z))^{\alpha/d } $ with base differential order
 $d = 2$ and $  \lambda = rd/\alpha $ }\label{TableGens2}
\end{table}

\begin{table}\centering
\begin{tabular}{lll}
\hline
\hline
$p$ & Coefficients $ \beta_j , 0\le j \le p+d-1 $ & with     $ d = 3 $ and $ \lambda = rd/\alpha $ \\
\hline
\hline
1 & $\beta_0= 1 $
 & $\beta_1= -3 $\\
 & $\beta_2= 3 $
 & $\beta_3= -1 $\\
\hline
2 & $\beta_0= - \lambda + \frac{5}{2} $
 & $\beta_1= 4 \lambda - 9 $\\
 & $\beta_2= - 6 \lambda + 12 $
 & $\beta_3= 4 \lambda - 7 $\\
 & $\beta_4= - \lambda + \frac{3}{2} $\\
\hline
3 & $\beta_0= \frac{\lambda^{2}}{2} - 3 \lambda + \frac{17}{4} $
 & $\beta_1= - \frac{5 \lambda^{2}}{2} + 14 \lambda - \frac{71}{4} $\\
 & $\beta_2= 5 \lambda^{2} - 26 \lambda + \frac{59}{2} $
 & $\beta_3= - 5 \lambda^{2} + 24 \lambda - \frac{49}{2} $\\
 & $\beta_4= \frac{5 \lambda^{2}}{2} - 11 \lambda + \frac{41}{4} $
3 & $\beta_5= - \frac{\lambda^{2}}{2} + 2 \lambda - \frac{7}{4} $\\
\hline
4 & $\beta_0= - \frac{\lambda^{3}}{6} + \frac{7 \lambda^{2}}{4} - \frac{35 \lambda}{6} + \frac{49}{8} $
 & $\beta_1= \lambda^{3} - 10 \lambda^{2} + 31 \lambda - 29 $\\
 & $\beta_2= - \frac{5 \lambda^{3}}{2} + \frac{95 \lambda^{2}}{4} - \frac{137 \lambda}{2} + \frac{461}{8} $
 & $\beta_3= \frac{10 \lambda^{3}}{3} - 30 \lambda^{2} + \frac{242 \lambda}{3} - 62 $\\
 & $\beta_4= - \frac{5 \lambda^{3}}{2} + \frac{85 \lambda^{2}}{4} - \frac{107 \lambda}{2} + \frac{307}{8} $
 & $\beta_5= \lambda^{3} - 8 \lambda^{2} + 19 \lambda - 13 $\\
 & $\beta_6= - \frac{\lambda^{3}}{6} + \frac{5 \lambda^{2}}{4} - \frac{17 \lambda}{6} + \frac{15}{8} $\\
\hline
5 & $\beta_0= \frac{\lambda^{4}}{24} - \frac{2 \lambda^{3}}{3} + \frac{23 \lambda^{2}}{6} - \frac{28 \lambda}{3} + \frac{967}{120} $
 & $\beta_1= - \frac{7 \lambda^{4}}{24} + \frac{9 \lambda^{3}}{2} - \frac{295 \lambda^{2}}{12} + \frac{111 \lambda}{2} - \frac{638}{15} $\\
 & $\beta_2= \frac{7 \lambda^{4}}{8} - 13 \lambda^{3} + \frac{135 \lambda^{2}}{2} - 142 \lambda + \frac{3929}{40} $
 & $\beta_3= - \frac{35 \lambda^{4}}{24} + \frac{125 \lambda^{3}}{6} - \frac{1235 \lambda^{2}}{12} + \frac{1219 \lambda}{6} - \frac{389}{3} $\\
 & $\beta_4= \frac{35 \lambda^{4}}{24} - 20 \lambda^{3} + \frac{565 \lambda^{2}}{6} - 176 \lambda + \frac{2545}{24} $
 & $\beta_5= - \frac{7 \lambda^{4}}{8} + \frac{23 \lambda^{3}}{2} - \frac{207 \lambda^{2}}{4} + \frac{185 \lambda}{2} - \frac{268}{5} $\\
 & $\beta_6= \frac{7 \lambda^{4}}{24} - \frac{11 \lambda^{3}}{3} + \frac{95 \lambda^{2}}{6} - \frac{82 \lambda}{3} + \frac{1849}{120} $
 & $\beta_7= - \frac{\lambda^{4}}{24} + \frac{\lambda^{3}}{2} - \frac{25 \lambda^{2}}{12} + \frac{7 \lambda}{2} - \frac{29}{15} $\\
\hline
\hline
\end{tabular}
\caption{Coefficients for polynomial $ p(z) = (\beta_0 + \beta_1 z+\cdot +\beta_{p+d-1} z^{p+d-1}  ) $ for generator $ W_{p,d}(z) = (p(z))^{\alpha/d } $ with base differential order
 $d = 3$ and $  \lambda = rd/\alpha $ }\label{TableGens3}
\end{table}

\subsection{Finite difference formulas}\label{SecFiniteDiff}

The finite difference formulas for classical derivatives
comes in various shapes and flavours. Forward, backward and central finite differences are well-known. There are also finite difference forms that gives
derivatives at any specified grid point. We call them \emph{shifted finite difference forms}.
In addition, there are finite difference forms that give  derivatives at a point   between two adjacent  points. These are called staggered finite difference forms. There are compact finite difference forms that use  the minimum number of grid points for a specified order of accuracy. In contrast, there are also non-compact finite difference forms.

Our proposed algorithm gives all these forms for
specified choices of input parameters of $d, p$ and $r$. In the following, we list the input parameters and their output finite difference forms in some details.

\subsection{Compact finite difference forms}
Compact finite difference forms are those that use the minimum number of
grid points for a specified derivative order and order of approximation and   are commonly used in applications.
There are known compact finite difference forms such as forward (right), backward (left) and central (symmetric).
There are also other compact forms in shifted and staggered formulations.
The presented algorithm gives all these forms.

In the presented algorithm, the input parameters for compact finite difference forms for a classical derivative $\alpha$ is chosen to be equal to the base order $d$.
The required order of approximation can be independently chosen. The number of grid points for all the compact finite difference forms for a differential order $d$ and approximation order $p$ is $ N = p+d $.
All the different  forms are obtained through the choice of the shift parameter $r$ as given in Table \ref{TableShifts}.

\begin{table}[h]\centering
\begin{tabular}{cccccc}
\hline
$\alpha = d$ & Left & Right & Central  & Shifted & Staggered \\
\hline
Shift $ r $ &  0  & $ p+d-1$ & $ (p+d-1)/2 $ & any int & any real\\
\hline
\end{tabular}
\caption{Choice of shift parameters for compact finite difference forms} \label{TableShifts}
\end{table}

Table \ref{Tablecompact} lists some compact
finite difference forms with one example of each kind mentioned above. The coefficient of the function at the point where the derivative is intended is indicated within parentheses. This is given by the index $ c = r - [r] = $ the fractional part of $r$. The finite difference summation formula
starts with function index $r$, (that is  $f_r$ ) from the leftmost coefficient and decrements  by one
for each subsequent coefficient.  The sum is then divided by $h^\alpha$.

The error term takes the leading error coefficient from the explicit form for the error multiplied by $ h^p f^{(p)}(\xi) $, where $\xi$ is a point
between the grid points.

For example, the central compact finite difference form for the derivative of order $\alpha = 3 $ with base $ d = 3 $ and approximation order $p = 4$ and shift $r = 3$  in Table
\ref{Tablecompact} is read as
\[
\frac{d^3}{dx^3}f_0 =  \frac{1}{h^3}
\left(
-\frac{1}{8}f_{3} +
f_{2} -
\frac{13}{8} f_{1} +
\frac{13}{8} f_{-1}
-f_{-2} +
\frac{1}{8} f_{-3}\right) -
\frac{7}{120} h^4 f^{(4)}(\xi) ,\]
where $ f_k = f(x_c + kh) $  and $ c $ is the fractional part of $ r $.

Note that $0\le c < 1$ with non-zero value for non-integer shift $r$ -- the staggered case and zero for integer shifts.
Thus, the last row of Table
\ref{Tablecompact} is read as
\begin{align*}
\frac{d^2}{dx^2}f_{0.5} & =  \frac{1}{h^2}
\left(
\frac{3}{16} f_{1.5}+
\frac{41}{48} f_{0.5} -
\frac{67}{24} f_{-1.5}+
\frac{19}{8} f_{-2.5} -
\frac{35}{48} f_{-3.5}+
\frac{5}{48} f_{-4.5} \right) \\ & +
\frac{341}{5760}  h^4 f^{(4)}(\xi).
\end{align*}

\begin{table}[h]\centering
\begin{tabular}{lllllll}
\hline\hline
$d$  &  $\alpha $   & $p$  & $ r $ &
Coefficients for Compact forms  & Error & Names\\
\hline\hline\\
1  &  1  &  3  &  0 & $
(\frac{11}{6}),
-3,
\frac{3}{2},
-\frac{1}{3}
$ & $
-\frac{1}{4}
$ & Left \\ \\
3  &  3  &  4  &  3 & $
-\frac{1}{8},
1,
-\frac{13}{8},
(0),
\frac{13}{8},
-1,
\frac{1}{8}
$ & $
\frac{-7}{120}
$ & Central \\ \\
2  &  2  &  4  &  1 & $
\frac{5}{6},
(-\frac{ 5}{4}),
-\frac{1}{3},
\frac{7}{6},
-\frac{1}{2},
\frac{1}{12}
$ & $
\frac{13}{180}
$ & Shifted \\ \\
3  &  3  &  4  &  6 & $
-\frac{15}{8},
13,
-\frac{307}{8},
62,
-\frac{461}{8},
29,
(-\frac{49}{8})
$ & $
-\frac{29}{15}
$ & Right \\ \\
2  &  2  &  4  &  1.5 & $
\frac{3}{16},
(\frac{41}{48}),
-\frac{67}{24},
\frac{19}{8},
-\frac{35}{48},
\frac{5}{48}
$ & $
\frac{341}{5760}
$ & Staggered \\ \\
\hline
\end{tabular}
\caption{Some Compact finite difference forms}\label{Tablecompact}
\end{table}

\subsection{Non-compact finite difference forms}
In contrast to the compact finite difference forms, non-compact
finite difference forms
use more grid points than that of the compact forms.
All the various finite difference forms  described in compact
forms are also available in the non-compact forms with the
same choices of the shift $r$.

The non-compact finite difference forms are obtained by choosing the base order $d$ such that $\alpha = \gamma d$ for some integer $\gamma > 1$.
The coefficients $\beta_j$ for the polynomial
$p(z) = \sum_{j=0}^{p+d} \beta_j z^j$ are first computed by using the unified explicit form in (\ref{Eq_NjDj}).
The generator for the non-compact form is then given by
$ W(z) = ( p(z) )^\gamma $ and is expanded to obtain the
non-compact coefficients. The expansion can be performed by the J. C. P. Miller recurrence formula as was done in the fractional derivative case.

For example, the non-compact finite difference formula for second derivative
($\alpha = 2 $) of approximation order $p =3 $ with $d = 1 \ne \alpha $
and shift $r = 1$ has the
coefficients and leading error coefficient computed from the unified formulation as
\[
\{ \beta_k \} = \left(
\frac{23}{24},
-\frac{7}{8},
-\frac{1}{8},
\frac{1}{24}
 \right), \quad R_3 = \frac{1}{12}.
\]

Therefore, the generator for the approximate formula is given
with power $ \gamma = \alpha/d = 2 $ as
\begin{align*}
W(z) &= \left(
\frac{23}{24}
-\frac{7}{8} z
-\frac{1}{8} z^2 +
\frac{1}{24} z^3
 \right)^2 \\
  & =
\frac{529}{576} -
\frac{161}{96} z +
\frac{101}{192} z^2 +
\frac{43}{144} z^3 -
\frac{11}{192} z^4 -
\frac{1}{96} z^5 +
\frac{1}{576}  z^6.
\end{align*}

Hence, the non-compact finite difference formula with its error term is given by
\begin{align*}
\frac{d^2}{dx^2} f_0 &=
\frac{1}{h^2} \left(\frac{529}{576} f_1 -
\frac{161}{96} f_0  +
\frac{101}{192} f_{-1} \right. \\ & \left.+
\frac{43}{144} f_{-2} -
\frac{11}{192}  f_{-3} -
\frac{1}{96}  f_{-4} +
\frac{1}{576}   f_{-5}\right)
 + \frac{1}{12} h^3 f^{(3)}(\xi).
\end{align*}

\section{Numerical tests}\label{SecApplications}

In order to  demonstrate the usefulness of the presented unified explicit formulation, we present here some numerical examples. Analysis of these
applications is beyond the scope of this paper and will be considered
in the future.

\subsection{Boundary value problem}

As a simple application, consider the boundary value problem:

\begin{equation}\label{eqIntOrderBVP}
\frac{d^2}{dx^2}u(x)=f(x),\quad a\leq x  \leq b, \qquad u(a)=u_a, \quad  u(b)=u_b .
\end{equation}

Consider the computational domain $a=x_0<x_1<\cdots <x_N=b$ with uniform discretization of the domain $[a, b]$ with subinterval size $h=x_{i+1}-x_i$. Usually, the second derivative is approximated by the central difference formula of order 2 for each internal point. That is, with $u_i=u(x_i)$ and $f_i=f(x_i)$,

\begin{equation}\label{eqBVPAtinternalGridPoint}
\frac{d^2}{dx^2} u(x_i)=f(x_i), \qquad i=1, 2, \cdots, N-1
\end{equation}
is approximated as

\begin{equation}\label{eqCentralDifferenceEq}
\frac{1}{h^2} \left(u_{i-1}-2u_i+u_{i+1}\right)=f_i+O(h^2), \qquad i=1, 2, \cdots, N-1.
\end{equation}

The resulting tridiagonal matrix equation is solved, after imposing the boundary conditions, for the discrete solution  $u_i$ which approximates $u(x_i)$  with  $O(h^2)$ error.

Now that we have the explicit form (\ref{eqnBetaNumerator-Denominator}) for any order of accuracy $p$ with $N= p+d$ grid points, we may approximate (\ref{eqBVPAtinternalGridPoint})  by a higher order approximation through  (\ref{Eq:Difference_form}). Since   $\alpha = 2$ ,  we choose $d=2$, so that $\alpha/d=1$.  We have from (\ref{eqW(z)d}), $W(z)= \beta_0+\beta_1z+\cdots+\beta_{N-1}z^{N-1}$, where $\beta_j=\beta_j(r), j = 0,1,\cdots, N-1$ with shift $r$. Each equation in (\ref{eqBVPAtinternalGridPoint}) is approximated  with shift $r=i$.
The approximation of (\ref{eqBVPAtinternalGridPoint}) is then given by

\begin{equation}\label{eqProposedDifference}
\frac{1}{h^2}\sum_{j=0}^{N-1}\beta_j(i)u_j=f_i+O(h^{N-1}), \qquad i=1, 2, \cdots, N-1.
\end{equation}

Note that the order of this approximation is $N-d$ and hence, the higher the
number of nodes, the higher the order of approximation. Matrix form of this system is given by $BU=F$, where we set
$U=\left[\hat{u}_0, \hat{u}_1, \cdots, \hat{u}_N\right]^T$,  $F=\left[f_1, f_2, \cdots, f_{N-1}\right]^T$, and the matrix $B$ is given by $$B_{i, j}=\beta_{j}(i), \qquad i=1, 2, \cdots, N-1, \qquad j=0, 1, \cdots, N.$$

The entries of the matrix $B$ are given by (\ref{eqnBetaNumerator-Denominator}) and are automatically computed by the algorithms for $\beta_j$. Imposing boundary conditions, the system reduces to
 \begin{equation}\label{eqRead-to-solve system}
 \hat{B}\hat{u}=F-B_0u_0-B_Nu_N,
 \end{equation}

where    $\hat{B}=\left[B_1, B_2, \cdots, B_{N-1}\right]$, where  $B_k$  are columns of $B$,
$\hat{U}=\left[\hat{u}_1,\hat{u}_2, \cdots, \hat{u}_{N-1} \right]^T$, and $\hat{u}_k \approx u_k$.
Solving (\ref{eqRead-to-solve system}) yields the approximate solution.

\begin{exmp}\rm
	We apply the above numerical schemes to the following   boundary value problem:
\begin{align}\label{eq:Example1}
	\frac{d^2}{dx^2}u(x) & = -\sin(x), \quad -1<x <1,\\
    u(-1) & =\sin(-1), \quad u(1)= \sin(1). \nonumber
\end{align}
for which the exact solution is $u(x)=\sin(x)$.
\end{exmp}

In order to compare and see the effectiveness of the automation and accuracy of the
unified explicit forms, the problem is solved first by using the well-known central difference formula of order 2 for the second derivative and then by using the
highest order compact finite difference formulas that use the full set of nodal values of the solution vector.

Table \ref{tab:Maximum Errors} shows the maximum absolute errors for the central difference method (\ref{eqCentralDifferenceEq}) and the proposed method in (\ref{eqProposedDifference}). As the order of the proposed method increases with the size of the discrete domain, precision errors occur for higher order approximations as the accuracy of the solution surpasses the machine precision--see orders 30, 62 and 126 in Table \ref{tab:Maximum Errors}.  In  order to make sure that these  errors are indeed due to accuracy overwhelming the machine precision, we computed the same methods with high precision arithmetic with 300 decimal digits accuracy. In this case, the    errors disappear and the absolute maximum errors continue to reduce.

Although we solve the resulting matrix equation by the standard classical solvers, the solution process may be efficiently performed seeing that the matrix has a centro-symmetric structure.  This has not been considered here and is thus open.

\begin{table}[ht]
	\centering\footnotesize

	\begin{tabular}{|l|l|l|r|l|l|}
		
		\hline
		&Step size&  Order 2&\multicolumn{3}{c|}{Proposed Higher Order Explicit Methods}\\
		\cline{4-6}
		$N$&$h$&Central&Order&Prec.  15 digits&Prec.  300 digits\\
		\hline
		4&0.5&2.3370e-01&2&0.00123815&0.00123815\\
		8&0.25&5.3029e-02&6&1.85125e-07&1.85125e-07\\
		16&0.125&1.2951e-02&14&5.55112e-15&2.02095e-17\\
		32&0.0625&3.2190e-03&30&4.71049&7.17654e-42\\
		64&0.03125&8.0358e-04&62&(Precision error)&4.76176e-100\\
		128&0.015625&2.0082e-04&126&(Precision error)&2.71135e-235\\
		\hline

	\end{tabular}
	\caption{Maximum errors for  Example 1: Order 2 Central difference and the proposed method }	
	\label{tab:Maximum Errors}
\end{table}

\subsection{Fractional boundary value problem}

The use of the difference approximations for fractional derivatives
is an ongoing research. Therefore, the use of the unified explicit formulation is yet to be explored. Nevertheless, we present here a simple example of fractional boundary value problem that uses the difference formula with base differential order  $d = 2$.

\begin{exmp}\label{Eq:Example2}
Consider the fractional boundary value problem
\begin{align}
D^\alpha y(x) &= \frac{\Gamma {(4+x)} }{6} x^3 , 0 < x < 1, \qquad 1<\alpha <2 \\
u(0) &= 0, \quad u(1) = 1. \nonumber
\end{align}
\end{exmp}
The exact solution is $ y(x) = x^{3+\alpha }. $

The domain $[0,1] $ is discretized with points
$ x_j = jh, j = 0,1,\cdots, N $ and the fractional derivative
is approximated by the order $ p =2 $ approximation with base order $d = 2 $
and shift $r = 1$ from Table \ref{TableGens2}. The weights $w_k $ we computed by using the J.C.P. Miller relation in (\ref{Eq:Miller}). The resulting system
\[
  BY = F
\]
  is solved after imposing the boundary conditions.

\begin{table}[ht]
	\centering\footnotesize

	\begin{tabular}{|l|l|l|r|l|l|l|}
		
		\hline
		& \multicolumn{2}{c|}{$\alpha=1.33 $} &  \multicolumn{2}{c|}{$\alpha=1.34 $}
&  \multicolumn{2}{c|}{$\alpha=1.6 $} \\
		\cline{2-7}
		$N$ & Error & Order & Error & Order & Error & Order \\
		\hline
		8   &  2.7809e-02 &  -- & 2.7809e-02 & -- & 1.7798e-02 & -- \\
		16  &  6.1767e-03 &  2.1707e+00  & 5.7018e-03 & 2.1871e+00 &4.4935e-03  & 7.0856e-05\\
		32  &  1.4568e-03 & 2.0840e+00  & 1.3175e-03  & 2.1136e+00  & 1.1292e-03 & 1.9925e+00\\
		64  &  3.5663e-04 & 2.0303e+00  & 3.1401e-04  & 2.0690e+00 &2.8309e-04  &1.9960e+00 \\
		128 &  9.3875e-05 & 1.9256e+00  & 7.6700e-05  & 2.0335e+00 & 7.0856e-05 & 1.9983e+00\\
        256 &  5.7965e-05 & 6.9557e-01  & 1.9031e-05  & 2.0109e+00 & 1.7725e-05 & 1.9991e+00\\
        512 &  1.6220e-05 & 1.8374e+00  & 4.7521e-06  & 2.0017e+00 & 4.4327e-06 & 1.9996e+00 \\
        1024&  3.7781e-06 & 2.1021e+00  & 1.1880e-06  & 2.0001e+00 & 1.1083e-06 & 1.9998e+00 \\
        2048&  1.3291e-06 & 1.5072e+00  & 2.9696e-07  & 2.0002e+00 & 2.7710e-07 & 1.9999e+00 \\
        4096&  1.8479e+05 & -3.7017e+01 & 7.4056e-08  & 2.0036e+00 & 6.9267e-08 & 2.0002e+00 \\
		\hline
	\end{tabular}
	\caption{Maximum errors and convergence orders for  Example 2 using the unified explicit formula with $d = 2, p = 2$}	
	\label{tab:Example2}
\end{table}

Table \ref{tab:Example2} lists the maximum errors and convergence orders for $\alpha = 1.33,1.34, 1.6 $.

The order 2 accuracy is obtained for $  \alpha >4/3  $   which has been confirmed by
numerical tests with more $\alpha $ values and can be justified
by the convergence of the series of the generator as follows.
The generator $W_{2,2}(z)$ for $ p = 2, d =2 $ is given by (see Table \ref{TableGens2})
\[
W_{2,2}(z) = (\beta_0 +\beta_1 z + \beta_2 z^2 + \beta_3 z^3 )^{\alpha/2}
= (1-z)^\alpha (\beta_0 + \beta_3 z)^{\alpha/2},
\]
where $\beta_0 = 2 - \lambda , \beta_1 = -5 + 3\lambda ,\beta_2 = 4- 3\lambda, \beta_3 = -1 + \lambda$ with $\lambda = rd/\alpha = 2r/\alpha $.
The series expansion   of $W_{2,2}(z)$ converges when the series of both factors
$ (1-z)^\alpha $ and $ (\beta_0+\beta_3 z)^{\alpha/2} $  converges -- that is, when $ |z| \le 1 $ and $ |\beta_3/\beta_0 z| < 1 $.

This gives
$
|\beta_3/\beta_0 |   < 1 $ and resolves to  $ 4r/3 < \alpha $.
Moreover, for real valued expansion of the second factor, it requires $\beta_0 > 0$ which leads to  $ r < \alpha  $.  For $ r = 1 $, these give  $ \alpha > 4/3.$

Further properties  and analysis of the higher order approximation formulas for fractional derivatives are open to be considered.

\section{Conclusion}\label{SecConclusion}
A unified explicit form  for the difference formulas for classical and fractional order derivatives was presented. This formulation gives various finite difference formulas such as left, right, central, shifted, compact and staggered finite difference formulas of an arbitrary order for classical derivatives. The formulation also gives Gr\"unwald type difference approximation of any order for fractional derivatives.

From this  unified formulation, some new type of difference formulas for fractional derivatives were also obtained.

Basic examples of applications were presented to demonstrate the effect of these difference formulas for classical and fractional boundary value problems.
More properties, analysis and applications of these formulas are open for research.

We believe that the unified explicit formula will contribute to the automation of solution process of differential problems with minimal manual intervention.

\appendix

\section{Vandemonde type Determinants}\label{App:Vander}

In this appendix, we establish the determinant of the
matrix obtained in Section \ref{SecUnfiedExplictForm}.

\begin{lem}
It is known that, for a finite sequence of $q+1$ parameters  $x_0, x_1, \cdots, x_q$, the determinant of the   Vandermonde matrix $V(x_0, x_1, \cdots, x_q)$
  of size $(q+1)$  is given by
  \[
  |V|=|V(x_0, x_1, \cdots, x_q)| = \left|
  \begin{array}{ccccc}
    1   & 1   & 1   & \cdots  & 1 \\
    x_0 & x_1 & x_2 & \cdots  & x_q \\
    x_0^2 & x_1^2 & x_2^2  & \cdots & x_q^2 \\
    \cdots & \cdots & \cdots  & \cdots & \cdots \\
    x_0^q & x_1^q & x_2^q &  \cdots & x_q^q
  \end{array} \right|
  = \prod_{0\le i < j \le q} (x_j - x_i).
  \]
\end{lem}

\begin{lem}
Let a variant of the Vandermonde matrix with $q+1$ parameters
be defined as
\[
U_k(x_0, x_1, \cdots, x_q) =
\left[
  \begin{array}{cccc}
    1   & 1     & \cdots &  1 \\
    x_0 & x_1 &   \cdots & x_q \\
    x_0^2 & x_1^2 &   \cdots & x_q^2 \\
    \cdots & \cdots  & \cdots & \cdots \\
    x_0^{k-1} & x_1^{k-1} &   \cdots & x_q^{k-1}\\
    x_0^{k+1} & x_1^{k+1} &   \cdots & x_q^{k+1}\\

    \cdots & \cdots  & \cdots & \cdots \\
    x_0^{q+1} & x_1^{q+1} &   \cdots & x_q^{q+1}
  \end{array} \right],
\]
where, from the $k$-th row onward, the powers of the parameters are increased by one.
Then, the determinant of $U_k$, which is of size $(q+1)$, is given by
  \begin{equation}\label{VVendermode}
  |U_k|
  = \prod_{0\le i < j \le q} (x_j - x_i)
  \left(  \sum_{\sigma \in B} \prod_{l\in \sigma} x_l \right)
  = |V| \left(  \sum_{\sigma \in B} \prod_{l\in \sigma} x_l \right),
  \end{equation}
  where $B$ is the set of all combinations of  $q+1-k$ elements chosen
  from  the index set $ A = \{0,1,2,\cdots,q\}$  given by
  $B = \{  \sigma = (m_1,m_2,\cdots, m_{q+1-k}) : m_k \in A \}$.
\end{lem}

\begin{proof}
  It is clear that $|U_k|$ is a homogeneous polynomial of total degree $ 1+2+\cdots+(k-1)+(k+1)+\cdots+ (q+1) = (q+1)(q+2)/2 - k$.
  Moreover,
  for $i\ne j$, by replacing $x_i$ with $x_j$, we have $|U_k|=0$ and thus $(x_i - x_j) $ are factors of $|U_k|$. There are $q(q+1)/2$ such factors as in the Vandermonde determinant $|V|$. Hence,
  \[
  |U_k|= \prod_{0\le i < j \le q} (x_j - x_i) P(x_0, x_1, \cdots, x_q),
  \] where $P$ is a homogeneous polynomial of total degree
  $(q+1)(q+2)/2 -k - q(q+1)/2 = q + 1 - k$.

  Since there are $q+1$ parameters,
  $ P $ will have
  $ \left(\begin{array}{c} q+1 \\ q+1 -k \end{array} \right) $
  terms of products
  of combinations of $q+1-k$  elements chosen from the $q+1$ parameters. Thus, $P$ has the form
  \[
  P = C \left(  \sum_{\sigma \in A } \prod_{l\in\sigma} x_l \right).
   \]
  The homogeneity  of $P$ forces the constant coefficients of the product terms to be the same.
  Equating one of the terms of $|U_k|$, we see that $C=1$.
\end{proof}

\begin{lem}
The determinant $ | V_j( {\bf b})| $  in (\ref{eqBetabyCramerRule}) of Theorem \ref{ThmBetaSolution} is given by
\[
|V_{j}(\mathbf{d})| = (-1)^{j+d}d!\prod_{\begin{smallmatrix}
	m>n\\ m,n \ne j
	\end{smallmatrix}}(\lambda_{m}-\lambda_{n})N_j,
\]
\end{lem}
where $N_j$ is given in (\ref{Eq_NjDj}).

\begin{proof}
Since the matrix $ V_{j}(\mathbf{d}) $ is obtained from the Vandermonde matrix $V$ by replacing the $ j^{th} $ column by the vector $ {\bf d} $, the parameter $ \lambda_j $ is removed from the list $ \{ \lambda_0, \lambda_2, \cdots, \lambda_{N-1}\}.$
Performing the determinant evaluation with the $j^{th} $ column, we have
\[
|V_{j}(\mathbf{d})| = (-1)^{j+d}d!
\left|
  \begin{array}{cccc}
    1   & 1     & \cdots &  1 \\
    \lambda_0 & \lambda_1 &   \cdots & \lambda_q \\
    \lambda_0^2 & \lambda_1^2 &   \cdots & \lambda_q^2 \\
    \cdots & \cdots  & \cdots & \cdots \\
    \lambda_0^{d-1} & \lambda_1^{d-1} &   \cdots & \lambda_q^{d-1}\\
    \lambda_0^{d+1} & \lambda_1^{d+1} &   \cdots & \lambda_q^{d+1}\\

    \cdots & \cdots  & \cdots & \cdots \\
    \lambda_0^{q+1} & \lambda_1^{q+1} &   \cdots & \lambda_q^{q+1}
  \end{array} \right|.
\]
The determinant is in the form of the variant determinant in Lemma 2. Hence, we have
\begin{align*}
|V_{j}(\mathbf{d})| &= (-1)^{j+d}d! |V| \left(  \sum_{\sigma \in B} \prod_{l\in \sigma} \lambda_l \right) \\
&= (-1)^{j+d}d! |V| \sum_{\begin{smallmatrix}
		0\leq m_1<m_2<\cdots<m_{p-1}\leq N-1\\m_i \ne j, 1\le i \le p-1
		\end{smallmatrix}} \prod_{k=0}^{p-1}\lambda_{m_k} ,
\end{align*}
where $\sigma$ are the combinations of $ N-1-d= p-1 $ elements from the $N-1$ parameter set
$A_j = \{ \lambda_m : 0\le m \le N-1  \}  \setminus \{\lambda_j \} $.

Now, the last sum-product term is $N_j$  and
 $|V| = \prod_{\begin{smallmatrix}
	m>n\\ m,n \ne j
	\end{smallmatrix}}(\lambda_{m}-\lambda_{n}) $  ends the proof.
\end{proof}

\section{Elementary symmetric polynomials}\label{App:ESP}
The elementary symmetric polynomials  on a finite set $X$ of parameters are defined as follows:

\begin{definition}
Let $X = \{x_0,x_1,\cdots, x_{N-1} \} $ be a set of $N$ parameters and $B_k$ be the set of all combinations $\sigma $  of  $k$ parameters chosen from $X$. Then
\[
S(X,k) = \sum_{\sigma \in B_k} \prod_{l \in \sigma} x_l
= \sum_{0 \le i_1 < i_2 <\cdots <i_k \le N-1} x_{i_1}x_{i_2}\cdots x_{i_k}.
\]
\end{definition}

It is interesting to note that the monomial $$L(x) = \prod_{k=0}^{N-1} (x + x_k) $$ has the expanded form given by
$$
L(x) = \sum_{k=0}^{N} S(X,k)x^{N-k}
$$
with $S(X,0) = 1$.
The elementary symmetric polynomials can thus be obtained from the coefficients of powers of the expanded $L(x)$.

The product in $L(x) $ can be performed recursively by multiplying  the factors $(x+x_k)$ successively as follows:
\begin{enumerate}
\item Let $L_{-1}(x) = 1$.

\item Define $L_k(x) = L_{k-1}(x)(x + x_k), k = 0,1,\cdots, N-1$.

\item If $L_{k-1}(x) = p_0 + p_1 x  +\cdots + p_{k-1} x^{k-1} $ and
$L_k(x)  = q_0 + q_1 x  + \cdots + q_{k} x^{k}$, then,
 by setting
$p_{-1} = p_{k} = 0$, the coefficients $q_m $ are related by
\begin{equation}\label{addfactor}
q_m = p_{m-1} + x_k p_{m}, \quad  m = 0,1,2 \cdots ,k.
\end{equation}
\end{enumerate}
Conversely, one can obtain $L_{k-1}(x)$ from  $L_k(x)$ through
\begin{equation}\label{removefactor}
p_{m-1} = q_{m} - x_k p_{m}, \quad  m = k,k-1, \cdots ,2,1,0,
\end{equation}
where $q_{-1}$ is the remainder when dividing $L_k(x)$ by $(x+x_k)$ which is zero since $(x+x_k)$
is a factor of $L_k(x)$.  Hence, the recursion can be computed only for
$m = k,k-1, \cdots ,2,1$.

\bibliographystyle{elsarticle-num-names}

\bibliography{NasirNafabib}

\end{document}